\documentclass[11pt]{amsart}

\usepackage{amssymb}
\usepackage{mathrsfs}
\usepackage{amsfonts}
\usepackage{latexsym,amsmath,amsthm,amssymb,amsfonts}
\usepackage[usenames]{color}
\usepackage{xspace,colortbl}
\usepackage{graphicx}
\usepackage{tipa}

\newcommand{\be}{\begin{equation}}
\newcommand{\ee}{\end{equation}}
\newcommand{\beq}{\begin{eqnarray}}
\newcommand{\eeq}{\end{eqnarray}}

\newtheorem{prop}{Proposition}[section]

\newtheorem{remark}[prop]{Remark}

\def\begeq{\begin{equation}}
\def\endeq{\end{equation}}

\def\odot{\setbox0=\hbox{$\bigcirc$}\relax \mathbin {\hbox
to0pt{\raise.5pt\hbox to\wd0{\hfil $\wedge$\hfil}\hss}\box0 }}

\numberwithin{equation} {section}

\numberwithin{equation}{section}
\newtheorem{theorem}{\bf Theorem}[section]

\newtheorem{definition}[theorem]{\bf Definition}
\newtheorem{lemma}[theorem]{\bf Lemma}

\newtheorem{corollary}[theorem]{\bf Corollary}

\allowdisplaybreaks

\begin{document}
\title[Prescribed Weingarten curvature equations in warped product manifolds]
 {Prescribed Weingarten curvature \\ equations in warped product manifolds}

\author{
 Ya Gao,~~Chenyang Liu,~~Jing Mao$^{\ast}$}

\address{
Faculty of Mathematics and Statistics, Key Laboratory of Applied
Mathematics of Hubei Province, Hubei University, Wuhan 430062, China
}

\email{Echo-gaoya@outlook.com, 1109452431@qq.com, jiner120@163.com}

\thanks{$\ast$ Corresponding author}

\date{}
\maketitle
\begin{abstract}
In this paper, under suitable settings, we can obtain the existence
of solutions to a class of prescribed Weingarten curvature equations
in \emph{warped product manifolds} of special type by the standard
degree theory based on the \emph{a priori estimates} for the
solutions. This is to say that the existence of closed hypersurface
(which is graphic with respect to the base manifold and whose $k$-th
Weingarten curvature satisfies some constraint) in a given warped
product manifold of special type can be assured.
\end{abstract}

\maketitle {\it \small{{\bf Keywords}: Prescribed Weingarten
curvature equations, $k$-convex, starshaped, warped product
manifolds.

 }

}

%%%%%%%%%%%%%%%%%%%%%%%%%%%%%%%%%%%%%%%%%%%%%%%%%%%%%%%%%%%%%%%%%%%%%%%%%%%%%%%%%%%%%%%%%%%%%%%% ½éÉÜ
\section{Introduction} \label{S1}

Throughout this paper, let $(M^{n},g)$ be a compact Riemannian
$n$-manifold with the metric $g$, and let $I$ be an (unbounded or
bounded) interval in $\mathbb{R}$. Clearly,
$\bar{M}:=I\times_{f}M^{n}$ is actually the $(n+1)$-dimensional
warped product manifold (sometimes, for simplicity, just say
\emph{warped product}) endowed with the following metric
\begin{eqnarray} \label{wpm}
\bar{g}=dt^{2}+f^{2}(t)g,
\end{eqnarray}
where $f:I\rightarrow\mathbb{R}^{+}$ is a positive differential
function defined on $I$. Given a differentiable function
$u:M^{n}\rightarrow I$, its graph actually corresponds to the
following graphic hypersurface
\begin{eqnarray} \label{gr-1}
\mathcal{G}=\{X(x)=(u(x),x)|x\in M^{n}\}
\end{eqnarray}
in $\bar{M}$. Equivalently, we can say that $\mathcal{G}$ is graphic
w.r.t. \emph{the base manifold} $M^{n}$. Denote by $\bar\nabla$, $D$
the Riemannian connections on $\bar{M}$ and $M^{n}$, respectively.
Let $\{e_{i}\}_{i=1,2,\cdots,n}$ be an orthonormal frame field in
$M^{n}$. Then one can find an orthonormal frame field
$\{\bar{e}_{\alpha}\}_{\alpha=0,1,\cdots,n}$ in $\bar{M}$ such that
$\bar{e}_{i}=(1/f)e_{i}$, $1\leq\alpha=i\leq n$ and
$\bar{e}_{0}=\partial/\partial t$. The existence of the frame fields
can always be assured in the tangent space of a prescribed point.
 Denote by\footnote{~Clearly, for accuracy, here $D_{i}u$
should be $D_{e_{i}}u$. In the sequel, without confusion and if
needed, we prefer to simplify covariant derivatives like this. In
this setting, $u_{ij}:=D_{j}D_{i}u$, $u_{ijk}:=D_{k}D_{j}D_{i}u$
mean $u_{ij}=D_{\partial_{j}}D_{\partial_{i}}u$ and
$u_{ijk}=D_{\partial_{k}}D_{\partial_{j}}D_{\partial_{i}}u$,
respectively. We will also simplify covariant derivatives on
$\mathcal{G}$ and $\bar{M}$ similarly if necessary.}
$u_{i}:=D_{i}u$, $u_{ij}:=D_{j}D_{i}u$, and
$u_{ijk}:=D_{k}D_{j}D_{i}u$ the covariant derivatives of $u$ w.r.t.
the metric $g$.  Clearly, the tangent vectors of $\mathcal{G}$ are
given by
\begin{eqnarray*}
X_{i}=(D u,1)=e_{i}+u_{i}\partial/\partial
t=f\bar{e}_{i}+u_{i}\bar{e}_0, \qquad i=1,2,\ldots,n.
\end{eqnarray*}
Let $\langle\cdot,\cdot\rangle$ be the inner product w.r.t. the
metric $\bar{g}$. Then the induced metric $\widetilde{g}$ on
$\mathcal{G}$ has the form
\begin{equation*}\label{g_{ij}}
\widetilde{g}_{ij}=\langle
X_{i},X_{j}\rangle=f^2\delta_{ij}+u_{i}u_{j},
\end{equation*}
its inverse is given by
\begin{equation*}\label{g^{ij}}
\widetilde{g}^{ij}=\frac{1}{f^2}\left(\delta^{ij}-\frac{u^i  u^j
}{f^{2}+|D u|^2}\right),
\end{equation*}
where $u^{i}=g^{ij}u_{j}=\delta^{ij}u_{j}$ and $|D u|^2=u^{i}u_{i}$.
Of course, in this paper we use the Einstein summation convention --
repeated superscripts and subscripts should be made
summation\footnote{~In this setting, repeated Latin letters should
be made summation from $1$ to $n$.}. The outward unit normal vector
field of $\mathcal{G}$ is given by
\begin{eqnarray*}\label{nu}
\nu=\frac{1}{\sqrt{f^{2}+|D u|^2}}\left(f\frac{\partial}{\partial
t}-u^{i}f^{-1}e_{i}\right)=\frac{1}{\sqrt{f^{2}+|D
u|^2}}\left(f\bar{e}_0-u^{i}\bar{e}_i\right),
\end{eqnarray*}
and the component $h_{ij}$ of the second fundamental form $A$ of
$\mathcal{G}$ is computed as follows
\begin{equation}\label{h_{ij}}
h_{ij}=-\langle\bar{\nabla}_{X_{j}}X_{i},\nu\rangle
=\frac{1}{\sqrt{f^{2}+|D
u|^2}}\left(-fu_{ij}+2f'u_{i}u_{j}+f^{2}f'\delta_{ij}\right).
\end{equation}
One can also see \cite[Subsection 2.2]{clw} for the computations of
the above geometric quantities.
 Denote by $\lambda_{1},\lambda_{2},\ldots,\lambda_{n}$
the principal curvatures of $\mathcal{G}$, which are actually the
eigenvalues of the matrix $(h_{ij})_{n\times n}$ w.r.t. the metric
$\widetilde{g}$. The so-called \emph{$k$-th Weingarten curvature} at
$X(x)=(u(x),x)\in\mathcal{G}$ is defined as
 \begin{eqnarray}  \label{kwc}
\sigma_{k}(\lambda_{1}, \lambda_{2}, \cdots,
\lambda_{n})=\sum\limits_{1\leq i_{1}<i_{2}<\cdots<i_{k}\leq
n}\lambda_{i_{1}}\lambda_{i_{2}}\cdots\lambda_{i_{k}}.
 \end{eqnarray}
$V=f(u)\frac{\partial}{\partial t}$ is the position vector
field\footnote{~In $\mathbb{R}^{n+1}$ or the hyperbolic
$(n+1)$-space $\mathbb{H}^{n+1}$, there is no need to define the
vector field $V$ since these two spaces are two-points homogeneous
and global coordinate system can be set up, and then $X(x)$ can be
seen as the position vector directly. } of hypersurface
$\mathcal{G}$ in $\bar{M}$, and clearly, for any $x\in M^{n}$,
$V|_{x}$ is a one-to-one correspondence with $X(x)$. Let $\nu(V)$ be
the outward unit normal vector field along the hypersurface
$\mathcal{G}$ and $\lambda(V)=(\lambda_{1}, \lambda_{2}, \cdots,
\lambda_{n})$ be the principal curvatures of $\mathcal{G}$ at $V$.
Define the annulus domain $\bar{M}^{+}_{-}\subset\bar{M}$ as follows
 \begin{eqnarray*}  \label{annd}
\bar{M}^{+}_{-}:=\{(t,x)\in\bar{M}|r_{1}\leq t\leq r_{2}\}
 \end{eqnarray*}
with $r_{1}<r_{2}$. In this paper, we consider the following
Weingarten curvature equation
 \begin{eqnarray} \label{main equation}
 \sigma_{k}(\lambda(V))=\sum\limits_{l=0}^{k-1}\alpha_{l}(u(x),x)\sigma_{l}(\lambda(V)),
 \quad \forall V\in\mathcal{G}, \quad 2\leq k\leq n,
 \end{eqnarray}
where $\{\alpha_{l}(u(x),x)\}_{l=0}^{k-1}$ are given smooth
functions defined on $\mathcal{G}$. The $k$-th Weingarten curvature
$\sigma_{k}(\lambda(V))$ is also called $k$-th mean curvature.
Besides, when $k=1$, $2$ and $n$, $\sigma_{k}(\lambda(V))$
corresponds to the mean curvature, the scalar curvature and the
Gaussian curvature of $\mathcal{G}$ at $V$.

 We also need the
following conception:
\begin{definition}
For $1\leq k\leq n$, let $\Gamma_{k}$ be a cone in $ \mathbb{R}^{n}$
determined by
\begin{eqnarray*}
\Gamma_{k}=\{\lambda\in\mathbb{R}^{n}|\sigma_{l}(\lambda)>0,
~~l=1,2,\ldots,k\}.
\end{eqnarray*}
A smooth graphic hypersurface $\mathcal{G}\subset\bar{M}$ is called
$k$-admissible if at every position vector $V\in\mathcal{G}$,
 $(\lambda_{1},\lambda_{2},\ldots,\lambda_{n})\in\Gamma_{k}$.
\end{definition}

For the Eq. (\ref{main equation}), we can prove the following:

\begin{theorem} \label{maintheorem}
Let $M^{n}$ be a compact Riemannian $n$-manifold ($n\geq3$) and
$\bar{M}=I\times_{f}M^{n}$, with the metric (\ref{wpm}), be the
warped product manifold defined as before. Assume that the warping
function $f$ is positive differential, $f'>0$, and
$\alpha_{l}(u(x),x)\in C^{\infty}(I\times M^{n})$ are positive
functions for all $0\leq l\leq k-1$. Suppose that
\begin{eqnarray} \label{as-1}
\sigma_{k}(e)\left(\frac{f'}{f}\right)^{k}\geq\sum\limits_{l=0}^{k-1}\alpha_{l}(u,x)\sigma_{l}(e)\left(\frac{f'}{f}\right)^{l}
\qquad for~u\geq r_{2},
\end{eqnarray}
\begin{eqnarray}  \label{as-2}
\sigma_{k}(e)\left(\frac{f'}{f}\right)^{k}\leq\sum\limits_{l=0}^{k-1}\alpha_{l}(u,x)\sigma_{l}(e)\left(\frac{f'}{f}\right)^{l}
\qquad for~0<u\leq r_{1},
\end{eqnarray}
and
\begin{eqnarray}  \label{as-3}
\frac{\partial}{\partial
u}\left[f^{k-l}(u)\alpha_{l}(u,x)\right]\leq0 \quad for~
r_{1}<u<r_{2},
\end{eqnarray}
where $[r_{1},r_{2}]\subset I$, $e=(1,1,\cdots,1)$. Then there
exists a smooth $k$-admissible, closed graphic hypersurface
$\mathcal{G}$ contained in the interior of the annulus
$\bar{M}^{+}_{-}$ and satisfying the Eq. (\ref{main equation}).
\end{theorem}

\begin{remark} \label{remark-1}
\rm{ (1) The $k$-admissible and the graphic properties of the
hypersurface $\mathcal{G}$ make sure that the Eq. (\ref{main
equation}) is a single scalar second-order elliptic PDE of the
graphic function $u$, which is the cornerstone of the a prior
estimates given below. If furthermore $M^{n}$ is convex, then
$M^{n}$ is diffeomorphic to $\mathbb{S}^{n}$ (i.e., the Euclidean
unit $n$-sphere), $\mathcal{G}$ is also a graphic hypersurface over
$\mathbb{S}^{n}$ and should be starshaped. In this setting, Theorem
\ref{maintheorem} degenerates into the following:
\begin{itemize}
 \item \underline{\textbf{FACT 1}}. \emph{Under the assumptions of Theorem \ref{maintheorem}, if furthermore $M^{n}$ is convex, then there
exists a smooth $k$-admissible, starshaped closed hypersurface
$\mathcal{G}$ contained in the interior of the annulus
$\bar{M}^{+}_{-}$ and satisfying the Eq. (\ref{main equation}).}
\end{itemize}
(2) We refer readers to, e.g., \cite[Appendix A]{mdw}, \cite[pp.
204-211 and Chapter 7]{bon} for an introduction to the notion and
properties of warped product manifolds. Submanifolds in warped
product manifolds have nice geometric properties and interesting
results can be expected -- see, e.g., several nice eigenvalue
estimates for the drifting Laplacian and the nonlinear $p$-Laplacian
on minimal submanifolds in warped product manifolds of prescribed
type have
been shown in \cite[Sections 3-5]{lmwz}.\\
 (3) The Eq. (\ref{main equation}) is actually a combination of
 elementary symmetric functions of eigenvalues of a given
 $(0,2)$-type tensor. Equations of this type are important not only in the study of PDEs but also
  in the study of many important geometric problems. For instance,
  if $\lambda(V)$ in the Eq. (\ref{main equation}) were replaced by
  eigenvalues of the Hessian $D^{2}u$ of a graphic function $u$ defined over a
  bounded $(k-1)$-convex domain $\Omega\subset\mathbb{R}^{n}$, Krylov
  \cite{nk} studied the corresponding PDE
\begin{eqnarray} \label{ME-2}
 \sigma_{k}(D^{2}u(x))=\sum\limits_{l=0}^{k-1}\alpha_{l}(x)\sigma_{l}(D^{2}u(x)), \quad \forall x\in\Omega,
\end{eqnarray}
   with a prescribed
  Dirichlet boundary condition (DBC for short) and coefficients $\alpha_{l}(x)\geq0$
  for all $0\leq l\leq k-1$, and observed that the natural admissible cone to make equation elliptic
  is $\Gamma_{k}$; recently, Guan-Zhang \cite{gz} showed that comparing with Krylov's this observation, for
  the admissible solution of Eq. (\ref{ME-2}) with prescribed DBC in the
  sense that $\lambda(D^{2}u)\in\Gamma_{k-1}$,
  there is no sign requirement for the coefficient
function of $\alpha_{k-1}(x)$. Moreover, they also investigated the
solvability of the following fully nonlinear elliptic equation
 \begin{eqnarray*}
\sigma_{k}(D^{2}u+uI)=\sum\limits_{l=0}^{k-1}\alpha_{l}(x)\sigma_{l}(D^{2}u+uI),
\quad \forall x\in\mathbb{S}^n,
 \end{eqnarray*}
 for some unknown function $u:\mathbb{S}^{n}\rightarrow\mathbb{R}$ defined over $\mathbb{S}^{n}$, where $\alpha_{l}(x)$, $0\leq l\leq k-2$,
 are positive functions;
   Fu-Yau \cite{fy1,fy2} proposed an
  equation of this type in
   the study of the Hull-Strominger system in theoretical
  physics; Phong-Picard-Zhang
  investigated the Fu-Yau equation and its generalization in series
  works \cite{ppz1,ppz2,ppz3}. Recently, inspired by Krylov's and
  Guan-Zhang's works \cite{gz,nk}, Chen-Shang-Tu \cite{cst}
  considered the following equation
   \begin{eqnarray} \label{ME-2}
\sigma_{k}(\kappa(X))=\sum\limits_{l=0}^{k-1}\alpha_{l}(X)\sigma_{l}(\kappa(X)),
\quad \forall X\in\mathcal{M}\subset\mathbb{R}^{n+1}, \qquad 2\leq
k\leq n
   \end{eqnarray}
on an embedded, closed starshaped $n$-hypersurface $\mathcal{M}$,
$n\geq3$, where $\kappa(X)$ are principal curvatures of
$\mathcal{M}$ at $X$, and $\alpha_{l}(x)$, $0\leq l\leq k-1$,
 are positive functions defined over $\mathcal{M}$. Under the
 $k$-convexity for $\mathcal{M}$ and several other growth
 assumptions (see \cite[Theorem 1.1]{cst}), they can show the existence of solutions to Eq.
 (\ref{ME-2}). This result has already been generalized by Shang-Tu
 \cite{st1} to the situation that the ambient space $\mathbb{R}^{n+1}$ was replaced by
 the hyperbolic space $\mathbb{H}^{n+1}$.
   }
\end{remark}

If $M^{n}=\mathbb{S}^{n}$, $I=(0,\ell)$ with $0<\ell\leq\infty$,
putting a one-pint compactification topology by identifying all
pairs $\{0\}\times\mathbb{S}^{n}$ with a single point $p^{\ast}$ to
$\bar{M}$ (see, e.g., \cite[page 705]{fmi} for this notion) and
requiring that $f(0)=0$, $f'(0)=1$, then the warped product manifold
$\bar{M}$ becomes the spherically symmetric manifold
$\widetilde{M}:=[0,\ell)\times_{f}\mathbb{S}^{n}$. The single point
$p^{\ast}$ is called the \emph{base point} of $\widetilde{M}$.
Applying \textbf{FACT 1} in Remark \ref{remark-1} directly, one has:

\begin{corollary} \label{coro-1}
Under the assumptions of Theorem \ref{maintheorem} with additionally
$M^{n}=\mathbb{S}^{n}$, $I=(0,\ell)$ with $0<\ell\leq\infty$,
one-pint compactification topology imposed, $f(0)=0$ and $f'(0)=1$,
then there exists a smooth $k$-admissible, starshaped (w.r.t. the
base point $p^{\ast}$), closed hypersurface $\mathcal{G}$ contained
in the interior of the annulus $\bar{M}^{+}_{-}\subset\widetilde{M}$
and satisfying the Eq. (\ref{main equation}).
\end{corollary}

\begin{remark}
\rm{ (1) If furthermore the warping function $f$ satisfies
$f''(t)+Kf(t)=0$ for some constant $K$, i.e. the Jacobi equation,
then
\begin{eqnarray*}
f(t)=\left\{
\begin{array}{lll}
\sin(\sqrt{K}t)/\sqrt{K}, \quad \qquad & K>0,~\ell=\pi/\sqrt{K},\\
t, \quad \qquad & K=0,~\ell=\infty,\\
f(t)=\sinh(\sqrt{-K}t)/\sqrt{-K}, \quad \qquad & K<0,~\ell=\infty,
\end{array}
\right.
\end{eqnarray*}
and moreover, in this setting, $\widetilde{M}$ corresponds to
$\mathbb{S}^{n+1}(1/\sqrt{K})$ (i.e., the Euclidean $(n+1)$-sphere
with radius $1/\sqrt{K}$) with the antipodal point of $p^{\ast}$
missed, $\mathbb{R}^{n+1}$ and $\mathbb{H}^{n+1}(K)$ (i.e., the
hyperbolic $(n+1)$-space with constant curvature $K<0$),
respectively. From this, one can see that spherically symmetric
manifolds cover space forms as a special case and actually they were
called \emph{generalized space forms} by Katz and Kondo \cite{KK}.
\\
 (2) Clearly, our Corollary \ref{coro-1} covers Chen-Shang-Tu's and
 Shang-Tu's main results in \cite{cst,st1} (mentioned in (3) of Remark \ref{remark-1}) as special cases. \\
 (3) Spherically symmetric manifolds have nice symmetry in
 non-radial direction, which leads to the fact that one can use
 this kind of manifolds as model space in the study of
 comparison theorems. In fact, Prof. J.
 Mao and his collaborators have used spherically symmetric manifolds
 as model space to successfully obtain Cheng-type eigenvalue
 comparison theorems for the first Dirichlet eigenvalue of the Laplacian on complete manifolds with
 radial (Ricci and sectional) curvatures bounded, Escobar-type
 eigenvalue comparison theorem for the first nonzero Steklov
 eigenvalue of the Laplacian on complete manifolds with radial
 sectional curvature bounded from above, heat kernel and volume comparison
 theorems for complete manifolds with suitable curvature
 constraints, and so on -- see \cite{fmi,m1-1,m1-2,m1,ywmd} for details.
}
\end{remark}

This paper is organized as follows. In Section \ref{S2}, we will
list some useful formulas including several basic properties of
$\sigma_{k}$, structure equations for hypersurfaces in warped
product manifolds. A priori estimates (including $C^0$, $C^1$ and
$C^2$ estimates) for solutions to the Eq. (\ref{main equation}) will
be shown continuously in Sections \ref{S3}-\ref{S5}. In Section
\ref{S6}, by applying the degree theory, together with the a priori
estimates obtained, we can prove the existence of solutions to
prescribed Weingarten curvature equations of type (\ref{main
equation}).

\section{Some useful formulae} \label{S2}

Except the setting of notations in Section \ref{S1}, denote by
$\bar\nabla$, $\nabla$ the Riemannian connections on $\bar{M}$ and
$\mathcal{G}$, respectively. The curvature tensors in $\bar{M}$ and
$\mathcal{G}$ will be denoted by $\bar{R}$ and $R$, respectively.
Let $\{E_{0}=\nu,E_{1},\cdots,E_{n}\}$ be an orthonormal frame field
in
 $\mathcal{G}$ and
$\{\omega_{0},\omega_{1},\cdots,\omega_{n}\}$ is its associated dual
frame field. The connections forms $\{\omega_{ij}\}$ and curvature
forms $\{\Omega_{ij}\}$ in $\mathcal{G}$ satisfy the structure
equations
\begin{eqnarray*}
d\omega_{i}-\sum\limits_{i}\omega_{ij}\wedge \omega_{j}=0,\quad \omega_{ij}+\omega_{ji}=0,
\end{eqnarray*}
\begin{eqnarray*}
d\omega_{ij}-\sum\limits_{k}\omega_{ik}\wedge \omega_{kj}=\Omega_{ij}=-\frac{1}{2}\sum\limits_{k,l}R_{ijkl}\omega_{k}\wedge \omega_{l}.
\end{eqnarray*}
The coefficients $h_{ij}$, $1\leq i,j\leq n$, of the second
fundamental form are given by Weingarten equation
\begin{eqnarray}\label{Weingarten-eq}
\omega_{i0}=\sum\limits_{j}h_{ij}\omega_{j}.
\end{eqnarray}
The covariant derivatives of the second fundamental form $h_{ij}$ in
$\mathcal{G}$ are given by
\begin{eqnarray*}
\sum\limits_{k}h_{ijk}\omega_{k}=dh_{ij}+\sum\limits_{l}h_{il}\omega_{lj}+
\sum\limits_{l}h_{lj}\omega_{li},
\end{eqnarray*}
\begin{eqnarray*}
\sum\limits_{l}h_{ijkl}\omega_{l}=dh_{ijk}+\sum\limits_{l}h_{ljk}\omega_{li}+
\sum\limits_{l}h_{ilk}\omega_{lj}+\sum\limits_{l}h_{ljl}\omega_{lk}.
\end{eqnarray*}
The Codazzi equation is
\begin{eqnarray}\label{codazzi-eq}
h_{ijk}-h_{ikj}=-\bar{R}_{0ijk},
\end{eqnarray}
and the Ricci identity can be obtained as follows:

\begin{lemma} (see also \cite[Lemma 2.2]{clw})
Let $X(x)$ be a point of $\mathcal{G}$ and
$\{E_{0}=\nu,E_{1},\cdots,E_{n}\}$ be an adapted frame field such
that each $E_{i}$ is a principal direction and $\omega^{k}_{i}=0$ at
$X(x)$. Let $(h_{ij})$ be the second quadratic form of
$\mathcal{G}$. Then, at the point $X(x)$, we have
\begin{eqnarray}\label{ricci-eq}
\begin{split}
h_{llii}=&h_{iill}-h_{lm}(h_{mi}h_{il}-h_{ml}h_{ii})-h_{mi}(h_{mi}h_{ll}-h_{ml}h_{li})\\
&+\bar{R}_{0iil;l}-2h_{ml}\bar{R}_{miil}+h_{il}\bar{R}_{0i0l}+h_{ll}\bar{R}_{0ii0}\\
&+\bar{R}_{0lil;i}-2h_{mi}\bar{R}_{mlil}+h_{ii}\bar{R}_{0l0l}+h_{li}\bar{R}_{0li0}.
\end{split}
\end{eqnarray}
\end{lemma}

As mentioned in Section \ref{S1}, one can suitably choose local
coordinates such that $\{e_{i}\}_{i=1,2,\cdots,n}$ is an orthonormal
frame field in $M^{n}$, and then one can find an orthonormal frame
field $\{\bar{e}_{\alpha}\}_{\alpha=0,1,\cdots,n}$ in $\bar{M}$ such
that $\bar{e}_{i}=(1/f)e_{i}$, $1\leq\alpha=i\leq n$ and
$\bar{e}_{0}=\partial/\partial t$. Correspondingly, the associated
dual frame field of $\{\bar{e}_{\alpha}\}_{\alpha=0,1,\cdots,n}$
should be $\{\theta_{\alpha}\}_{\alpha=0,1,\cdots,n}$ with
$\bar{\theta}_{i}=f\theta_{i}$, $1\leq i \leq n$, and
$\bar{\theta}_{0}=dt$. Clearly, $\{\theta_{i}\}_{i=1,\cdots,n}$ is
the dual frame field of the orthonormal frame field
$\{e_{i}\}_{i=1,2,\cdots,n}$. We have the following fact:

\begin{lemma} (see \cite{clw})
On the leaf $M_{t}$ of the warped product manifold
$\bar{M}=I\times_{f}M^{n}$, the curvature satisfies
\begin{eqnarray}\label{R_ijk0}
\bar{R}_{ijk0}=0
\end{eqnarray}
and the principal curvature is given by
\begin{eqnarray}\label{kappa}
\kappa(t)=\frac{f'(t)}{f(t)}
\end{eqnarray}
where the outward unit normal vector
$\bar{e}_{0}=\frac{\partial}{\partial t}$ is chosen for each leaf
$M_{t}$.
\end{lemma}

\begin{remark}
\rm{ In fact, the leaf $M_{t}$ can also be seen as a closed graphic
hypersurface in $\bar{M}$, which corresponds to the graph of some
constant function, i.e. $u=const.$. Besides, we refer readers to
\cite[Section 2]{clw} or \cite{PP} for the geometry of hypersurfaces
in warped product manifolds if necessary.
 }
\end{remark}

Consider two functions $\tau :\mathcal{G}\rightarrow \mathbb{R}$ and
$\Lambda :\mathcal{G}\rightarrow \mathbb{R}$ given by
\begin{eqnarray}\label{f-1}
\tau=f\langle\nu,\bar{e}_{0}\rangle=\langle V,\nu\rangle,\qquad
\Lambda=\int_{0}^{u} f(s) ds,
\end{eqnarray}
where $V=f\bar{e}_{0}=f\frac{\partial}{\partial t}$ is the position
vector field and $\nu$ is the outward unit normal vector field. Then
we have:

\begin{lemma}  \label{f-2} (see \cite{ajb})
The gradient vector fields of the functions $\tau$ and $\Lambda$ are
\vspace{0.25cm}
\begin{eqnarray}\label{g-la}
\nabla_{E_{i}}\Lambda=f\left<\bar{e}_{0},E_{i}\right>,
\end{eqnarray}
\begin{eqnarray}\label{g-ta}
\nabla_{E_{i}}\tau=\sum\limits_{j}\nabla_{E_{j}}\Lambda h_{ij},
\end{eqnarray}
and the second order derivatives of $\tau$ and $\Lambda$ are given
by
\begin{eqnarray}\label{d2-la}
\nabla^{2}_{E_{i},E_{j}}\Lambda=-\tau h_{ij}+f'g_{ij},
\end{eqnarray}
\begin{eqnarray}\label{d2-ta}
\nabla^{2}_{E_{i},E_{j}}\tau=-\tau\sum\limits_{k}h_{ik}h_{kj}+f'h_{ij}+\sum\limits_{k}(h_{ijk}+\bar{R}_{0ijk})\nabla_{E_{k}}\Lambda.
\end{eqnarray}
\end{lemma}

The following Newton-Maclaurin inequality will be used frequently
(see, e.g., \cite{mt1,t2}).

\begin{lemma}\label{NM-ieq}
Let $\lambda \in \mathbb{R}^{n}$. For $0\leq l\leq k\leq n,~~r>s\geq
0,~~k\geq r,~~l\geq s$, we have
\begin{eqnarray*}
k(n-l+1)\sigma_{l-1}(\lambda)\sigma_{k}(\lambda)\leq l(n-k+1)\sigma_{l}(\lambda)\sigma_{k-1}
\end{eqnarray*}
and
\begin{eqnarray*}
\left[\frac{\sigma_{k}(\lambda)/C_{n}^{k}}{\sigma_{l}(\lambda)/C_{n}^{l}}\right]^{\frac{1}{k-l}}
\leq
\left[\frac{\sigma_{r}(\lambda)/C_{n}^{r}}{\sigma_{s}(\lambda)/C_{n}^{s}}\right]^{\frac{1}{r-s}},
\qquad for~~ \lambda\in\Gamma_{k}.
\end{eqnarray*}
\end{lemma}

At end, we also need the following truth to ensure the ellipticity
of the Eq. \eqref{equation 1.1}.

\begin{lemma}\label{ellip-}
Let $\mathcal{G}=\{\left( u(x),x\right)|x\in M^{n}\}$ be a smooth
$(k-1)$-admissible closed hypersurface in $\bar{M}$ and
$\alpha_{l}(u,x)\geq 0$ for any $x\in M^{n}$ and $0\leq l\leq k-2$.
Then the operator
\begin{eqnarray*}
G\left(h_{ij}(V),u,x\right):=\frac{\sigma_{k}(\lambda(V))}{\sigma_{k-1}(\lambda(V))}
-\sum\limits_{l=0}^{k-2}\alpha_{l}(u,x)\frac{\sigma_{l}(\lambda(V))}{\sigma_{k-1}(\lambda(V))}
\end{eqnarray*}
is elliptic and concave with respect to $h_{ij}(V)$.
\end{lemma}

\begin{proof}
The proof is almost the same with the one of \cite[Proposition
2.2]{gz}, and we prefer to omit here.
\end{proof}

\section{$C^0$ estimate}  \label{S3}

We consider the family of equations for $0\leq t\leq 1$,
\begin{eqnarray}\label{equation 1.1}
\frac{\sigma_{k}(\lambda(V))}{\sigma_{k-1}(\lambda(V))}-\sum\limits_{l=0}^{k-2}t\alpha_{l}(u,x)\frac{\sigma_{l}(\lambda(V))}{\sigma_{k-1}(\lambda(V))}-\alpha_{k-1}(u,x,t)=0,
\end{eqnarray}
where
\begin{eqnarray*}
\alpha_{k-1}(u,x,t):=t\alpha_{k-1}(u,x)+(1-t)\varphi(u)\frac{\sigma_{k}(e)}{\sigma_{k-1}(e)}\frac{f'}{f},
\end{eqnarray*}
and $\varphi$ is a positive function defined on $I$ and satisfying
the following conditions:

(a) $\varphi(u)>0$;

(b) $\varphi(u)>1$ for $u \leq r_{1}$;

(c) $\varphi(u)<1$ for $u \geq r_{2}$;

(d) $\varphi'(u)<0$.

\begin{lemma}[\textbf{$C^{0}$ estimate}]\label{C0 estimate}
Assume that $0\leq \alpha_{l}(u,x)\in C^{\infty}(I\times M^{n}).$
Under the assumptions \eqref{as-1} and \eqref{as-2} mentioned in
Theorem \ref{maintheorem}, if $\mathcal{G}=\{(u(x),x)|x\in
M^{n}\}\subset \bar{M}$ is a smooth $(k-1)$-admissible, closed
graphic hypersurface satisfied the curvature equation
\eqref{equation 1.1} for a given $t\in[0,1]$, then
\begin{eqnarray*}
r_{1}\leq u(x) \leq r_{2},\qquad \forall x\in M^{n}.
\end{eqnarray*}
\end{lemma}
\begin{proof}
Assume that $u(x)$ attains its maximum at $x_{0}\in M^{n}$ and
$u(x_{0})\geq r_{2}$. Then from \eqref{h_{ij}}, one has
\begin{eqnarray*}
h^{i}_{j}=\frac{1}{v}\left[f'\delta^{i}_{j}+\frac{1}{v^{2}}\left(f'u_{j}u^{i}-fu^{i}_{j}\right)\right],
\end{eqnarray*}
where $v=\sqrt{f^{2}+|D u|^{2}}$, which implies
\begin{eqnarray*}
h^{i}_{j}(x_{0})=\frac{1}{f}\left(f'\delta^{i}_{j}-\frac{u^{i}_{j}}{f}\right)\geq\frac{f'}{f}\delta^{i}_{j}.
\end{eqnarray*}
Note that $\frac{\sigma_{k}}{\sigma_{k-1}}$ and $\frac{\sigma_{k-1}}{\sigma_{l}}$ with $0\leq l\leq k-2$ is concave in $\Gamma_{k-1}$. Thus,
\begin{eqnarray*}
\frac{\sigma_{k}}{\sigma_{k-1}}(h^{i}_{j})\geq\frac{\sigma_{k}}{\sigma_{k-1}}\left(\frac{f'}{f}\delta^{i}_{j}\right)+
\frac{\sigma_{k}}{\sigma_{k-1}}\left(-\frac{1}{f^{2}}u_{j}^{i}\right)
 \geq\frac{\sigma_{k}}{\sigma_{k-1}}\left(\frac{f'}{f}\delta^{i}_{j}\right).
\end{eqnarray*}
Therefore, it follows that
\begin{eqnarray*}
\frac{\sigma_{k}(\lambda(V))}{\sigma_{k-1}(\lambda(V))}\geq\frac{\sigma_{k}(e)}{\sigma_{k-1}(e)}\frac{f'}{f}.
\end{eqnarray*}
Similarly, one can get
\begin{eqnarray*}
\frac{\sigma_{l}(\lambda(V))}{\sigma_{k-1}(\lambda(V))}\leq\frac{\sigma_{l}(e)}{\sigma_{k-1}(e)}\left(\frac{f}{f'}\right)^{k-l-1}.
\end{eqnarray*}
Combining with the above two inequalities, we have
\begin{eqnarray*}
\frac{\sigma_{k}(e)}{\sigma_{k-1}(e)}\frac{f'}{f}-\sum\limits_{l=0}^{k-2}t\alpha_{l}(u,x)\frac{\sigma_{l}(e)}{\sigma_{k-1}(e)}\left(\frac{f}{f'}\right)^{k-l-1}\leq
\alpha_{k-1}(u,x,t).
\end{eqnarray*}
Clearly, if $t=0$, the above inequality is contradict with
\eqref{equation 1.1}. When $0<t\leq1$, we can obtain
\begin{eqnarray*}
\begin{split}
\alpha_{k-1}(u,x)&=\left(1-\frac{1}{t}\right)\varphi\frac{f'}{f}\frac{\sigma_{k}(e)}{\sigma_{k-1}(e)}+\frac{1}{t}\alpha_{k-1}(x,u,t)\\
&\geq
\left(\frac{1}{t}\frac{f'}{f}-\left(1-\frac{1}{t}\right)\varphi\frac{f'}{f}\right)\frac{\sigma_{k}(e)}{\sigma_{k-1}(e)}
-\sum\limits_{l=0}^{k-2}\alpha_{l}(u,x)\frac{\sigma_{l}(e)}{\sigma_{k-1}(e)}\left(\frac{f}{f'}\right)^{k-l-1}\\
&>\frac{f'}{f}\frac{\sigma_{k}(e)}{\sigma_{k-1}(e)}-\sum\limits_{l=0}^{k-2}\alpha_{l}(u,x)\frac{\sigma_{l}(e)}{\sigma_{k-1}(e)}\left(\frac{f}{f'}\right)^{k-l-1},
\end{split}
\end{eqnarray*}
which is contradict with
\begin{eqnarray*}
\frac{f'}{f}\frac{\sigma_{k}(e)}{\sigma_{k-1}(e)}-\sum\limits_{l=0}^{k-2}\alpha_{l}(u,x)\frac{\sigma_{l}(e)}{\sigma_{k-1}(e)}\left(\frac{f}{f'}\right)^{k-l-1}\geq
\alpha_{k-1}(u,x)
\end{eqnarray*}
 in view of \eqref{as-1} and the condition $\varphi(u)<1$ for $u\geq r_{2}$. This shows $\sup u\leq r_{2}$. Similarly,
 we can obtain $\inf u\geq r_{1}$ in view of \eqref{as-2} and the condition $\varphi(u)>1$ for $u\leq r_{1}$. Our proof is finished.
\end{proof}

Now, we can prove the following uniqueness result.

\begin{lemma}\label{uni-sol}
For $t=0$, there exists a unique admissible solution of the Eq.
\eqref{equation 1.1}, namely $\mathcal{G}_{0}=\{(u(x),x)\in
\bar{M}|u(x)=u_{0}\}$, where $u_{0}$ is the unique solution of
$\varphi(u_{0})=1$.
\end{lemma}

\begin{proof}
Let $\mathcal{G}_{0}$ be a solution of \eqref{equation 1.1}, and
then for $t=0$,
\begin{eqnarray*}
\frac{\sigma_{k}(\lambda(V))}{\sigma_{k-1}(\lambda(V))}-\varphi(u)\frac{\sigma_{k}(e)}{\sigma_{k-1}(e)}\frac{f'}{f}=0.
\end{eqnarray*}
Assume that $u(x)$ attains its maximum $u_{\mathrm{max}}$ at
$x_{0}\in M^{n}$. Then one has
\begin{eqnarray*}
\frac{\sigma_{k}(\lambda(V))}{\sigma_{k-1}(\lambda(V))} \geq
\frac{\sigma_{k}(e)}{\sigma_{k-1}(e)}\frac{f'}{f},
\end{eqnarray*}
which implies
\begin{eqnarray*}
\varphi(u_{\mathrm{max}}) \geq 1.
\end{eqnarray*}
Similarly, the minimum $u_{\mathrm{min}}$ of $u(x)$ satisfies
\begin{eqnarray*}
\varphi(u_{\mathrm{min}}) \leq 1.
\end{eqnarray*}
Since $\varphi$ is a decreasing function, we obtain
\begin{eqnarray*}
\varphi(u_{\mathrm{max}}) = \varphi(u_{\mathrm{min}}) = 1,
\end{eqnarray*}
which implies that $u(x_{0})=u_{0}$ for any $(u(x),x)\in
\mathcal{G}_{0}$, with $u_{0}$ the unique solution of
$\varphi(u_{0})=1$.
\end{proof}

\section{$C^1$ estimate}  \label{S4}

We can rewritten the Eq. \eqref{equation 1.1} as follows:
\begin{eqnarray*}
G(h_{ij}(V),u,x,t)=\frac{\sigma_{k}(\kappa(V))}{\sigma_{k-1}(\kappa(V))}-\sum\limits_{l=0}^{k-2}t\alpha_{l}(u,x)\frac{\sigma_{l}(\kappa(V))}{\sigma_{k-1}(\kappa(V))}=\alpha_{k-1}(u,x,t).
\end{eqnarray*}
For convenience, we will simplify notations as follows:
\begin{eqnarray*}
G_{k}(h_{ij}(V)):=\frac{\sigma_{k}(\lambda(V))}{\sigma_{k-1}(\lambda(V))},\qquad
G_{l}(h_{ij}(V))=:-\frac{\sigma_{l}(\lambda(V))}{\sigma_{k-1}(\lambda(V))},
\end{eqnarray*}
and
\begin{eqnarray*}
G^{ij}(\lambda(V)):=\frac{\partial G}{\partial h_{ij}},\quad
G^{ij,rs}(\lambda(V)):=\frac{\partial^{2} G}{\partial h_{ij}\partial
h_{rs}}.
\end{eqnarray*}

\begin{lemma}[\textbf{$C^{1}$ estimate}]\label{C1 estimate}
Assume that $k\geq 2$ and
\begin{eqnarray*}
\alpha_{l}(u,x)\geq c_{l}>0,\qquad \forall x\in M^{n}
\end{eqnarray*}
for $0\leq l\leq k-1$. Under the assumption \eqref{as-3}, if the
smooth $(k-1)$-admissible, closed graphic hypersurface $\mathcal{G}$
satisfies the Eq.\eqref{main equation} and $u$ has positive upper
and lower bounds, then there exists a constant $C$ depending on $n$,
$k$, $c_{l}$, $|\alpha_{l}|_{C^{1}}$, the $C^{0}$ bound of $f$ and
the curvature tensor $\bar{R}$, the minimum and maximum values of
$u$ such that
\begin{eqnarray*}
|\nabla u(x)|\leq C,\qquad \forall x\in M^{n}.
\end{eqnarray*}
\end{lemma}

\begin{proof}
First, we know from \eqref{nu} and \eqref{f-1} that
\begin{eqnarray*}
\tau=\frac{f^{2}(u)}{\sqrt{f^{2}(u)+|Du|^{2}}}.
\end{eqnarray*}
It is sufficient to obtain a positive lower bound of $\tau$. Define
\begin{eqnarray*}
\psi=-\log \tau+\gamma(\Lambda),
\end{eqnarray*}
where $\gamma(t)$ is a function chosen later. Assume that $x_{0}$ is
the maximum value point of $\psi$. If $V$ is parallel to the normal
direction $\nu$ of $\mathcal{G}$ at $x_{0}$, our result holds since
$\left<V,\nu\right>=|V|.$ So, we assume that $V$ is not parallel to
the normal direction $\nu$ at $x_{0}$, we may choose the local
orthonormal frame field $\{E_{1},\cdots,E_{n}\}$ on $\mathcal{G}$
satisfying
\begin{eqnarray*}
\left\langle V,E_{1}\right\rangle\neq 0\quad {\rm and} \quad
\left\langle V,E_{i}\right\rangle= 0,\quad\forall~~i\geq2.
\end{eqnarray*}
Then, we arrive at $x_{0}$,
\begin{eqnarray}\label{tau-i}
\tau_{i}=\tau\gamma'\Lambda_{i}
\end{eqnarray}
and
\begin{eqnarray*}
\begin{split}
\psi_{ii}=&-\frac{\tau_{ii}}{\tau}+\frac{(\tau_{i})^{2}}{\tau^{2}}+\gamma''\Lambda_{i}^{2}+\gamma'\Lambda_{ii}\\
=&-\frac{1}{\tau}\left(\sum\limits_{k}(h_{iik}+\bar{R}_{0iik})\Lambda_{k}+f'h_{ii}-\tau h_{ii}^{2}\right)\\
&+\left((\gamma')^{2}+\gamma''\right)\Lambda_{i}^{2}+\gamma'(f'-\tau h_{ii})
\end{split}
\end{eqnarray*}
in view of
\begin{eqnarray*}
\tau_{ii}=\sum\limits_{k}(h_{iik}+\bar{R}_{0iik})\left\langle
V,E_{k}\right\rangle+f'h_{ii}-\tau\sum\limits_{k} h_{ik}h_{ki}.
\end{eqnarray*}
By \eqref{g-la}, \eqref{g-ta} and \eqref{tau-i}, we have at $x_{0}$
\begin{eqnarray}\label{h-11}
h_{11}=\tau \gamma', \quad h_{1i}=0,\quad \forall~~i\geq2.
\end{eqnarray}
Therefore, we can rotate the coordinate system such that
$\{E_{i}\}_{i=1}^{n}$ are the principal curvature directions of the
second fundamental form $h_{ij}$, i.e. $h_{ij}=h_{ii}\delta_{ij}$.
Since $\Lambda_{1}=\left\langle V,E_{1}\right\rangle$,
$\Lambda_{i}=\left\langle V,E_{i}\right\rangle$ for any $i\geq2$.
So, we can get
\begin{eqnarray*}
\begin{split}
G^{ii}\psi_{ii}=&-\frac{f'}{\tau}G^{ii}h_{ii}-\frac{1}{\tau}G^{ii}(h_{ii1}+\bar{R}_{0ii1})\Lambda_{1}+G^{ii}h_{ii}^{2}\\
&+\left((\gamma')^{2}+\gamma''\right)G^{11}\Lambda_{1}^{2}+\gamma'G^{ii}(f'-\tau h_{ii}).
\end{split}
\end{eqnarray*}
Noting that
\begin{eqnarray*}
G^{ij}h_{ij}=G-\sum\limits_{l=0}^{k-2}(k-l)\alpha_{l}G_{l}=\alpha_{k-1}(u,x,t)-\sum\limits_{l=0}^{k-2}(k-l)\alpha_{l}G_{l}
\end{eqnarray*}
and
\begin{eqnarray*}
G^{ij}h_{ij1}=\nabla_{1}\alpha_{k-1}(u,x,t)-\sum\limits_{l=0}^{k-2}t\nabla_{1}\alpha_{l}G_{l},
\end{eqnarray*}
we conclude
\begin{eqnarray}\label{G-psi}
\begin{split}
G^{ii}\psi_{ii}=&\frac{\Lambda_{1}}{\tau}\left(-\nabla_{1}\alpha_{k-1}(u,x,t)+\sum\limits_{l=0}^{k-2}t\nabla_{1}\alpha_{l}G_{l}\right)\\
&+\frac{f'}{\tau}\left(-\alpha_{k-1}(u,x,t)+\sum\limits_{l=0}^{k-2}(k-l)\alpha_{l}G_{l}\right)+G^{ii}h_{ii}^{2}\\
&-\frac{1}{\tau}G^{ii}\bar{R}_{0ii1}\Lambda_{1}+\left((\gamma')^{2}+\gamma''\right)G^{11}\Lambda_{1}^{2}+\gamma'G^{ii}(f'-\tau h_{ii})\\
=&\frac{1}{\tau}\left(-\Lambda_{1}\nabla_{1}\alpha_{k-1}(u,x,t)-f'\alpha_{k-1}(u,x,t)\right)\\
&+\frac{1}{\tau}\sum\limits_{l=0}^{k-2}tG_{l}\left(\Lambda_{1}\nabla_{1}\alpha_{l}+f'(k-l)\alpha_{l}\right)+G^{ii}h_{ii}^{2}\\
&-\frac{1}{\tau}G^{ii}\bar{R}_{0ii1}\Lambda_{1}+\left((\gamma')^{2}+\gamma''\right)G^{11}\Lambda_{1}^{2}+\gamma'G^{ii}(f'-\tau h_{ii}).
\end{split}
\end{eqnarray}
Since $\left\langle V,E_{i}\right\rangle=0$ for $i=2,\cdots,n$, we
obtain
\begin{eqnarray*}
V=\left\langle V,E_{1}\right\rangle
E_{1}+\left<V,\nu\right>\nu=\Lambda_{1}E_{1}+\tau\nu,
\end{eqnarray*}
which results in
\begin{eqnarray*}
\Lambda_{1}\nabla_{1}\alpha_{l}(u,x)+(k-l)f'\alpha_{l}(u,x)=\bar{\nabla}_{V}\alpha_{l}(u,x)+(k-l)f'\alpha_{l}(u,x)-\tau\bar{\nabla}_{\nu}\alpha_{l}(u,x).
\end{eqnarray*}
We know from the assumption \eqref{as-3} that
\begin{eqnarray*}
\left[(k-l)f'\alpha_{l}(u,x)+\nabla_{V}\alpha_{l}(u,x)\right]=\left[(k-l)f'\alpha_{l}(u,x)+f\frac{\partial \alpha_{l}(u,x)}{\partial u}\right]\leq 0.
\end{eqnarray*}
Thus,
\begin{eqnarray}\label{ieq-1}
\Lambda_{1}\nabla_{1}\alpha_{l}(u,x)+(k-l)f'\alpha_{l}(u,x)\leq
-\tau\bar{\nabla}_{V}\alpha_{l}(u,x)
\end{eqnarray}
and
\begin{eqnarray}\label{ieq-2}
\begin{split}
\qquad
\Lambda_{1}&\nabla_{1}\alpha_{k-1}(u,x,t)+f'\alpha_{k-1}(u,x,t) \leq
(1-t)\varphi'\frac{\sigma_{k}(e)}{\sigma_{k-1}(e)}-\tau\bar{\nabla}_{V}\alpha_{k-1}(u,x,t).
\end{split}
\end{eqnarray}
Taking \eqref{ieq-1} and \eqref{ieq-2} into \eqref{G-psi}, we have at $x_{0}$\vspace{0.25cm}
\begin{eqnarray}\label{ieq-3}
\begin{split}
0\geq~~ &G^{ii}\varphi_{ii}\\
\geq~~ &G^{ii}h_{ii}^{2}+\left((\gamma')^{2}+\gamma''\right)G^{11}\Lambda_{1}^{2}+\gamma'G^{ii}(f'-\tau h_{ii})-\frac{1}{\tau}G^{ii}\bar{R}_{0ii1}\Lambda_{1}\\
&-t\sum\limits_{l=0}^{k-2}G_{l}\bar{\nabla}_{\nu}\alpha_{l}(u,x)-\frac{(1-t)}{\tau}\varphi'\frac{\sigma_{k}(e)}{\sigma_{k-1}(e)}+\bar{\nabla}_{\nu}\alpha_{k-1}(u,x,t)\\
=~~&G^{ii}\left(h_{ii}-\frac{1}{2}\gamma'\tau\right)^{2}+\left((\gamma')^{2}+\gamma''\right)G^{11}\Lambda_{1}^{2}+G^{ii}\left(\gamma'f'-\frac{1}{4}(\gamma')^{2}\tau^{2}\right)\\
&-\frac{1}{\tau}G^{ii}\bar{R}_{0ii1}\Lambda_{1}-t\sum\limits_{l=0}^{k-2}G_{l}\bar{\nabla}_{\nu}\alpha_{l}(u,x)-\frac{(1-t)}{\tau}\varphi'\frac{\sigma_{k}(e)}{\sigma_{k-1}(e)}+\bar{\nabla}_{\nu}\alpha_{k-1}(u,x,t).
\end{split}
\end{eqnarray}
Choosing
\begin{eqnarray*}
\gamma(t)=-\frac{\alpha}{t}
\end{eqnarray*}
for sufficiently large positive constant $\alpha$, we have
\begin{eqnarray*}
\gamma'(t)=\frac{\alpha}{t^{2}},\quad \gamma''(t)=-\frac{2\alpha}{t^{3}}.
\end{eqnarray*}
Therefore, \eqref{ieq-3} becomes
\begin{eqnarray}\label{ieq-4}
\begin{split}
0\geq G^{ii}\left(\gamma'f'-\frac{1}{4}(\gamma')^{2}\tau^{2}\right)
-c_{1}\left(\sum\limits_{l=0}^{k-2}|G_{l}|+1\right)
-\frac{1}{\tau}G^{ii}\bar{R}_{0ii1}\Lambda_{1}
\end{split}
\end{eqnarray}
in view of
\begin{eqnarray*}
(\gamma')^{2}+\gamma''\geq 0,
\end{eqnarray*}
 where $c_{1}$ is a positive constant depending on $|\alpha_{l}|_{C^{1}}$. Since $V=\left\langle V,E_{1}\right\rangle E_{1}+\left\langle V,\nu\right\rangle\nu$, we can
find that $V\perp{\mathrm{Span}}(E_{2},\cdots,E_{n})$, i.e., $V$ is
orthogonal with the subspace spanned by $E_{2},\cdots,E_{n}$. On the
other hand, $E_{1},\nu$ are orthogonal with
$\mathrm{Span}(E_{2},\cdots,E_{n})$. It is possible to choose
suitable coordinate system such that $\bar{E}_{1}\perp\mathrm{
Span}(E_{2},\cdots,E_{n})$, which implies that the pairs
$\{V,\bar{E}_{1}\}$ and $\{\nu,E_{1}\}$ lie in the same plane and
\begin{eqnarray*}
{\mathrm{Span}}(E_{2},...,E_{n})={\mathrm{Span}}(\bar{E}_{2},\cdots,\bar{E}_{n}),
\end{eqnarray*}
 where of course $\{\bar{E}_{0}=\bar{e}_0,\bar{E}_{1},\cdots,\bar{E}_{n}\}$ is a local orthonormal frame field in $\bar{M}$.
Therefore, we can choose
$E_{2}=\bar{E}_{2},\ldots,E_{n}=\bar{E}_{n}$, and then vectors $\nu$
and $E_{1}$ can be decomposed into
\begin{eqnarray*}
&&\nu=\left\langle\nu,\bar{e}_{0}\right\rangle\bar{e}_{0}+\left\langle\nu,\bar{E}_{1}\right\rangle\bar{E}_{1}=\frac{\tau}{f}\bar{e}_{0}+\left\langle\nu,\bar{E}_{1}\right\rangle\bar{E}_{1},\\
&&\qquad E_{1}=\left\langle
E_{1},\bar{e}_{0}\right\rangle\bar{e}_{0}+\left\langle
E_{1},\bar{E}_{1}\right\rangle\bar{E}_{1}.
\end{eqnarray*}
By \eqref{R_ijk0} and the fact $V=\Lambda_{1}E_{1}+\tau\nu$, we can
obtain
\begin{eqnarray}\label{eq-5}
\begin{split}
\bar{R}_{0ii1}&=\bar{R}(\nu,E_{i},E_{i},E_{1})\\
&=\frac{\tau}{f}\left\langle
E_{1},\bar{e}_{0}\right\rangle\bar{R}(\bar{e}_{0},
\bar{E}_{i},\bar{E}_{i},\bar{e}_{0})+\left\langle\nu,\bar{E}_{1}\right\rangle
\left\langle E_{1},\bar{E}_{1}\right\rangle\bar{R}(\bar{E}_{1},\bar{E}_{i},\bar{E}_{i},\bar{E}_{1})\\
&=\frac{\tau}{f}\left\langle
E_{1},\bar{e}_{0}\right\rangle\bar{R}(\bar{e}_{0},
\bar{E}_{i},\bar{E}_{i},\bar{e}_{0})-\tau\frac{\left\langle\nu,\bar{E}_{1}\right\rangle^{2}}
{\Lambda_{1}}\bar{R}(\bar{E}_{1},\bar{E}_{i},\bar{E}_{i},\bar{E}_{1})\\
&=\tau\left(\frac{1}{f}\left\langle
E_{1},\bar{e}_{0}\right\rangle\bar{R}(\bar{e}_{0},
\bar{E}_{i},\bar{E}_{i},\bar{e}_{0})-\frac{\left\langle\nu,\bar{E}_{1}\right\rangle^{2}}
{\Lambda_{1}}\bar{R}(\bar{E}_{1},\bar{E}_{i},\bar{E}_{i},\bar{E}_{1})\right),
\end{split}
\end{eqnarray}
where the third equality comes from $\left\langle
V,\bar{E}_{1}\right\rangle=0$. Substituting \eqref{eq-5} into
\eqref{ieq-4} yields
\begin{eqnarray}\label{ieq-6}
\begin{split}
0\geq G^{ii}\left(\gamma'f'-\frac{1}{4}(\gamma')^{2}\tau^{2}\right)
-c_{1}\left(\sum\limits_{l=0}^{k-2}|G_{l}|+1\right)
-c_{2}\sum\limits_{i}G^{ii}
\end{split}
\end{eqnarray}
where $c_{2}>0$ depends on the $C^{0}$ bound of $f$ and the
curvature tensor $\bar{R}$. To continue our proof, we need to
estimate $G_{l}$ for $0\leq l\leq k-2$. Let $P\in \mathbb{R}$ be a
fixed positive number.

$(I)$ If $\frac{\sigma_{k}}{\sigma_{k-1}}\leq P$, then we get from
$\alpha_{l}\geq c_{l}$ that
\begin{eqnarray*}
|G_{l}|=\frac{\sigma_{l}}{\sigma_{k-1}}\leq\frac{1}{\alpha_{l}}\left(\frac{\sigma_{k}}
{\sigma_{k-1}}+\alpha_{l}(u,x,t)\right)\leq c_{3}(P+1),
\end{eqnarray*}
where the constant $c_{3}>0$ depends on $c_{l}$,
$|\alpha_{l}|_{C^{0}}$.

$(II)$ If $\frac{\sigma_{k}}{\sigma_{k-1}}>P$, then by Lemma
\ref{NM-ieq}, one has
\begin{eqnarray*}
|G_{l}|=\frac{\sigma_{l}}{\sigma_{k-1}} \leq
\frac{\sigma_{l}}{\sigma_{l+1}}\cdot\frac{\sigma_{l+1}}{\sigma_{l+2}}
\cdot\cdot\cdot\cdot\frac{\sigma_{k-2}}{\sigma_{k-1}} \leq
c_{4}\left(\frac{\sigma_{k-1}}{\sigma_{k}}\right)^{k-1-l} \leq
P^{-(k-1-l)},
\end{eqnarray*}
where the positive constant $c_{4}$ depends on $k$.

So, $|G_{l}|$ can be bounded for any $0\leq l\leq k-2$. By the
definition of operator $G$ and a direct computation, we have
$\Sigma_{i}G^{ii}\geq \frac{n-k+1}{k}$, and so we can choose
sufficiently large $\alpha$ such that
\begin{eqnarray*}
0\geq G^{ii}\left[\gamma'f'-(\gamma')^{2}\tau^{2}\right].
\end{eqnarray*}
Thus,
\begin{eqnarray*}
\gamma'f'\leq (\gamma')^{2}\tau^{2},
\end{eqnarray*}
which means
\begin{eqnarray*}
\tau \geq c_{5}
\end{eqnarray*}
for some positive constant $c_{5}$ depending on $n$, $k$, $c_{l}$,
$|\alpha_{l}|_{C^1}$, the $C^{0}$ bound of $f$ and the curvature
tensor $\bar{R}$. The conclusion of Lemma \ref{C1 estimate} follows
directly.
\end{proof}

\begin{remark}
\rm{ After several careful revisions to the manuscript of this
paper, we prefer to number (by subscripts) nearly all the constants
in the $C^1$ and $C^{2}$ estimates, and we believe that this way can
  reveal the relations among constants clearly to readers.
}
\end{remark}

\section{$C^2$ estimates}  \label{S5}

This section devotes to the $C^2$ estimates. However, before that,
we need to make some preparations. First, we need the following
fact:

\begin{lemma}\label{C2-1}
Let $\mathcal{G}=\{\left( u(x),x\right)| x\in M^{n}\}$ be a
$(k-1)$-admissible solution of the Eq. \eqref{equation 1.1} and
assume that $\alpha_{l}(u,x)\geq 0$ for $0\leq l\leq k-1$. Then, we
have the following inequality
\begin{eqnarray*}
G^{ij}h_{ijpp}\geq \nabla_p\nabla_p\alpha_{k-1}(u,x,t)-
\sum\limits_{l=0}^{k-2}\frac{1}{1+\frac{1}{k+1-l}}\frac{t(\nabla_p\alpha_l)^2}{\alpha_l}G_l
-\sum\limits_{l=0}^{k-2}t\nabla_p\nabla_p\alpha_lG_l.
\end{eqnarray*}
\end{lemma}
\begin{proof}
Differentiating the Eq. \eqref{equation 1.1} once, we have
\begin{eqnarray*}
\nabla_p\alpha_{k-1}(u,x,t)=
G^{ij}h_{ijp}+\sum\limits_{l=0}^{k-2}t\nabla_p\alpha_{l}G_l.
\end{eqnarray*}
Differentiating the Eq. \eqref{equation 1.1} twice, we obtain
\begin{eqnarray*}
\nabla_p\nabla_p\alpha_{k-1}(u,x,t)=
G^{ij,rs}h_{ijp}h_{rsp}+G^{ij}h_{ijpp}
+2\sum\limits_{l=0}^{k-2}t\nabla_p\alpha_{l}G^{ij}_{l}h_{ijp}
+\sum\limits_{l=0}^{k-2}t\nabla_p\nabla_p\alpha_{l}G_l.
\end{eqnarray*}
Moreover, since the operator
$\left(\frac{\sigma_{k-1}}{\sigma_{l}}\right)^{\frac{1}{k-1-l}}$ is
concave for $0\leq l\leq k-2$, we have (see also (3.10) in
\cite{gz})
\begin{eqnarray*}
G^{ij,rs}h_{ijp}h_{rsp}\leq
\left(1+\frac{1}{k-1-l}\right)G_{l}^{-1}G^{ij}_{l}G^{rs}_{l}h_{ijp}h_{rsp}.
\end{eqnarray*}
Thus, in view that $G_{k}$ is concave in $\Gamma_{k-1}$, we have
\begin{eqnarray*}
\begin{split}
&\nabla_p\nabla_p\alpha_{k-1}(u,x,t)\\
\leq&\sum\limits_{l=0}^{k-2}t\alpha_{l}G_{l}^{ij,rs}h_{ijp}h_{rsp}+G^{ij}h_{ijpp}
+2\sum\limits_{l=0}^{k-2}t\nabla_p\alpha_{l}G^{ij}_{l}h_{ijp}
+\sum\limits_{l=0}^{k-2}t\nabla_p\nabla_p\alpha_{l}G_{l}\\
\leq&\sum\limits_{l=0}^{k-2}t\alpha_{l}G_{l}^{-1}\left(1+\frac{1}{k-1-l}\right)(G^{ij}_{l}h_{ijp})^{2}+G^{ij}h_{ijpp}+
2\sum\limits_{l=0}^{k-2}t\nabla_p\alpha_{l}G^{ij}_{l}h_{ijp}
+\sum\limits_{l=0}^{k-2}t\nabla_p\nabla_p\alpha_{l}G_{l}\\
=&\frac{k-l}{k-1-l}\sum\limits_{l=0}^{k-2}t\alpha_{l}G_{l}^{-1}\left(G^{ij}_{l}h_{ijp}+
\frac{1}{1+\frac{1}{k-1-l}}\frac{\nabla_p\alpha_{l}}{\alpha_{l}}G_{l}\right)^{2}+
\sum\limits_{l=0}^{k-2}\frac{1}{1+\frac{1}{k-1-l}}\frac{t(\nabla_p\alpha_{l})^{2}}
{\alpha_{l}}G_{l}\\
&+G^{ij}h_{ijpp}+\sum\limits_{l=0}^{k-2}t\nabla_p\nabla_p\alpha_{l}G_{l}\\
\leq&\sum\limits_{l=0}^{k-2}\frac{1}{1+\frac{1}{k-1-l}}\frac{t(\nabla_p\alpha_{l})^{2}}
{\alpha_{l}}G_{l}+G^{ij}h_{ijpp}+\sum\limits_{l=0}^{k-2}t\nabla_p\nabla_p\alpha_{l}G_{l},
\end{split}
\end{eqnarray*}
which completes the proof of Lemma \ref{C2-1}.
\end{proof}

We also need the following truth:

\begin{lemma}\label{C2-2}
Let $\mathcal{G}=\{\left( u(x),x\right)| x\in M^{n}\}$ be a
$(k-1)$-admissible solution of the Eq. \eqref{equation 1.1} with the
position vector $V$ in $\bar{M}$. We have the following equality
\begin{eqnarray*}
\begin{split}
&G^{ij}\tau_{ij}+\sum\limits_{k}\tau G^{ij}h_{ik}h_{kj}\\
=&\left(\nabla_p\alpha_{k-1}(u,x,t)-\sum\limits_{l=0}^{k-2}t\nabla_p\alpha_{l}G_{l}
+\sum\limits_{p}G^{ij}\bar{R}_{0ijp}\right)\left\langle
V,E_{p}\right\rangle
+f'\left(\alpha_{k-1}(u,x,t)-\sum\limits_{l=0}^{k-2}(k-l)t\alpha_{l}G_{l}\right).
\end{split}
\end{eqnarray*}
\end{lemma}

\begin{proof}
By Lemma \ref{f-2}, we have
\begin{eqnarray*}
\tau_{ij}=-\tau\sum\limits_{k}h_{ik}h_{kj}+f'h_{ij}
+\sum\limits_{k}(h_{ijk}+\bar{R}_{0ijk})\left\langle
V,E_{p}\right\rangle,
\end{eqnarray*}
which results in
\begin{eqnarray*}
G^{ij}\tau_{ij}=-\tau
G^{ij}\sum\limits_{k}h_{ik}h_{kj}+f'G^{ij}h_{ij}
+\sum\limits_{k}G^{ij}(h_{ijk}+\bar{R}_{0ijk})\left\langle
V,E_{p}\right\rangle.
\end{eqnarray*}
Note that
\begin{eqnarray*}
G^{ij}h_{ij}=G-\sum\limits_{l=0}^{k-2}(k-l)\alpha_{l}G_{l}
=\alpha_{k-1}(u,x,t)-\sum\limits_{l=0}^{k-2}(k-l)\alpha_{l}G_{l}
\end{eqnarray*}
and
\begin{eqnarray*}
G^{ij}h_{ijp}=\nabla_{p}\alpha_{k-1}(u,x,t)-\sum\limits_{l=0}^{k-2}t\nabla_{p}\alpha_{l}G_{l}.
\end{eqnarray*}
Thus,
\begin{eqnarray*}
\begin{split}
G^{ij}\tau_{ij}
=&\left(\nabla_p\alpha_{k-1}(u,x,t)-\sum\limits_{l=0}^{k-2}t\nabla_p\alpha_{l}G_{l}
+\sum\limits_{p}G^{ij}\bar{R}_{0ijp}\right)\left\langle V,E_{p}\right\rangle\\
&+f'\left(\alpha_{k-1}(u,x,t)-\sum\limits_{l=0}^{k-2}(k-l)t\alpha_{l}G_{l}\right)
-\sum\limits_{k}\tau G^{ij}h_{ik}h_{kj}.
\end{split}
\end{eqnarray*}
Therefore, we complete the proof.
\end{proof}

Now we begin to estimate the second fundamental form.

\begin{lemma}[\textbf{$C^{2}$ estimates}]\label{C2}
 Assume that $k\geq 2$ and
\begin{eqnarray*}
\alpha_{l}(u,x)\geq c_{l}>0,\qquad \forall x\in M^{n}
\end{eqnarray*}
for $0\leq l\leq k-1$. If the $k$-admissible, closed graphic
hypersurface $\mathcal{G}=\{\left( u(x),x\right)| x\in M^{n}\}$
 satisfies the Eq. \eqref{equation 1.1} with the position vector $V$ in $\bar{M}$, then there exists a constant $C$ depending on
$n$, $k$, $c_{l}$, $|\alpha_{l}|_{C^{2}}$, $|\nabla u|_{C^{0}}$, the
$C^{0}$, $C^{1}$ bounds of $f$ and the curvature tensor $\bar{R}$
such that for $1\leq i\leq n$, the principal curvatures of
$\mathcal{G}$ at $V$ satisfy
\begin{eqnarray*}
|\lambda_{i}(V)|\leq C,\qquad \forall x\in M^{n}.
\end{eqnarray*}
\end{lemma}

\begin{proof}
Since $k\geq 2$, $\mathcal{G}$ is $2$-admissible, for sufficiently
large $c_{6}$, one has
\begin{eqnarray*}
|\lambda_{i}|\leq c_{6}H,
\end{eqnarray*}
where the positive constant $c_{6}$ depends on $n$, $k$. So, we only
need to estimate the mean curvature $H$ of $\mathcal{G}$. Taking the
auxiliary function
\begin{eqnarray*}
W(x)=\log H-\log\tau.
\end{eqnarray*}
Assume that $x_{0}$ is the maximum point of $W$. Then at $x_{0}$,
one has
\begin{eqnarray}\label{W-1}
0=W_{i}=\frac{H_{i}}{H}-\frac{\tau_{i}}{\tau}
\end{eqnarray}
and
\begin{eqnarray}\label{W-2}
0\geq W_{ij}(x_{0})=\frac{H_{ij}}{H}-\frac{\tau_{ij}}{\tau}.
\end{eqnarray}
Choosing a suitable coordinate system $\{x^{1},x^{2},...,x^{n}\}$ in
the neighborhood of $X_{0}=\left(u(x_{0}),x_{0}\right)\in
\mathcal{G}$ such that the matrix $(h_{ij})_{n\times n}$ is diagonal
at $X_{0}$,  i.e., $h_{ij}=h_{ii}\delta_{ij}$. This implies at
$x_{0}$,
\begin{eqnarray}\label{W-ieq-1}
0\geq
G^{ij}W_{ij}(x_{0})=\sum\limits_{p=1}^{n}\frac{1}{H}G^{ii}h_{ppii}
-\frac{G^{ii}\tau_{ii}}{\tau}.
\end{eqnarray}
By \eqref{ricci-eq}, we can obtain
\begin{eqnarray*}
\begin{split}
h_{ppii}=&h_{iipp}+h_{pp}^{2}h_{ii}-h_{ii}^{2}h_{pp}+\bar{R}_{0iip;p}
+\bar{R}_{0pip;i}-2h_{pp}\bar{R}_{piip}\\
&+h_{ii}\bar{R}_{0i0i}+h_{pp}\bar{R}_{0ii0}
+h_{ii}\bar{R}_{0p0p}+h_{ii}\bar{R}_{0ii0}-2h_{ii}\bar{R}_{ipip}.
\end{split}
\end{eqnarray*}
Note that
\begin{eqnarray*}
G^{ij}h_{ij}=G-\sum\limits_{l=0}^{k-2}(k-l)\alpha_{l}G_{l}
=\alpha_{k-1}(u,x,t)-\sum\limits_{l=0}^{k-2}(k-l)\alpha_{l}G_{l}.
\end{eqnarray*}
So, we have
\begin{eqnarray*}
\begin{split}
\sum\limits_{p}G^{ii}h_{ppii}=&\sum\limits_{p}G^{ii}\left(h_{iipp}+\bar{R}_{0iip;p}
+\bar{R}_{0pip;i}\right)-\sum\limits_{p}h_{pp}G^{ii}\left(h_{ii}^{2}+
2\bar{R}_{piip}-\bar{R}_{0ii0}\right)\\
&+\sum\limits_{p}G^{ii}h_{ii}\left(h_{pp}^{2}-2\bar{R}_{ipip}
+\bar{R}_{0i0i}+\bar{R}_{0p0p}+\bar{R}_{0ii0}\right)\\
\geq& \sum\limits_{p}G^{ii}h_{iipp}+\left(|A|^{2}-c_{8}\right)
\left(\alpha_{k-1}(u,x,t)-\sum\limits_{l=0}^{k-2}(k-l)\alpha_{l}G_{l}\right)
-c_{7}\\
&-HG^{ii}\left(h_{ii}^{2}+c_{9}\right),
\end{split}
\end{eqnarray*}
where the positive constant $c_{7}$ depends on the $C^{1}$ bound of
the curvature tensor $\bar{R}$, the positive constants $c_{8}$,
$c_{9}$ depend on the $C^{0}$ bound of the curvature tensor
$\bar{R}$. Together with Lemma \ref{C2-1}, we know that
\eqref{W-ieq-1} becomes
\begin{eqnarray*}
\begin{split}
0\geq&\frac{1}{H}\sum\limits_{p=1}^{n}G^{ii}h_{iipp}+\frac{|A|^{2}-c_{8}}{H}
\left(\alpha_{k-1}(u,x,t)-\sum\limits_{l=0}^{k-2}(k-l)\alpha_{l}G_{l}\right)-\frac{G^{ii}\tau_{ii}}{\tau}\\
&-\frac{c_{7}}{H}\sum\limits_{i}G^{ii}-G^{ii}(h_{ii}^{2}+c_{9})\\
\geq&\frac{1}{H}\sum\limits_{p=1}^{n}\left(\nabla_p\nabla_p\alpha_{k-1}(u,x,t)-
\sum\limits_{l=0}^{k-2}\frac{1}{1+\frac{1}{k+1-l}}\frac{t(\nabla_p\alpha_l)^2}{\alpha_l}G_l
-\sum\limits_{l=0}^{k-2}t\nabla_p\nabla_p\alpha_lG_l\right)-\frac{G^{ii}\tau_{ii}}{\tau}\\
&+\frac{|A|^{2}-c_{8}}{H}\left(\alpha_{k-1}(u,x,t)-\sum\limits_{l=0}^{k-2}(k-l)t\alpha_{l}G_{l}\right)
-\frac{c_{7}}{H}\sum\limits_{i}G^{ii}-G^{ii}(h_{ii}^{2}+c_{9}).\\
\end{split}
\end{eqnarray*}
By Lemma \ref{C2-2}, the above inequality becomes
\begin{eqnarray*}
\begin{split}
0\geq&\frac{1}{H}\sum\limits_{p=1}^{n}\left(\nabla_p\nabla_p\alpha_{k-1}(u,x,t)-
\sum\limits_{l=0}^{k-2}\frac{1}{1+\frac{1}{k+1-l}}\frac{t(\nabla_p\alpha_l)^2}{\alpha_l}G_l
-\sum\limits_{l=0}^{k-2}t\nabla_p\nabla_p\alpha_lG_l\right)\\
&+\frac{|A|^{2}-c_{8}}{H}\left(\alpha_{k-1}(u,x,t)-\sum\limits_{l=0}^{k-2}(k-l)t\alpha_{l}G_{l}\right)
-\frac{c_{7}}{H}\sum\limits_{i}G^{ii}-G^{ii}(h_{ii}^{2}+c_{9})\\
&-\frac{1}{\tau}\left(\nabla_p\alpha_{k-1}(u,x,t)-\sum\limits_{l=0}^{k-2}t\nabla_p\alpha_{l}G_{l}
+\sum\limits_{p}G^{ii}\bar{R}_{0iip}\right)\left\langle V,E_{p}\right\rangle\\
&-\frac{f'}{\tau}\left(\alpha_{k-1}(u,x,t)-\sum\limits_{l=0}^{k-2}(k-l)t\alpha_{l}G_{l}\right)
+G^{ii}h_{ii}^{2}.
\end{split}
\end{eqnarray*}
Hence, we have
\begin{eqnarray*}
\begin{split}
0\geq&\frac{|A|^{2}}{H}\left(\alpha_{k-1}(u,x,t)-\sum\limits_{l=0}^{k-2}(k-l)t\alpha_{l}G_{l}\right)
-\left(\frac{c_{7}}{H}+c_{9}\right)\sum\limits_{i}G^{ii}-\frac{\left<V,E_{p}\right>}{\tau}
\sum\limits_{p}G^{ii}\bar{R}_{0iip}\\
&+\frac{1}{H}\sum\limits_{p=1}^{n}\left(\nabla_p\nabla_p\alpha_{k-1}(u,x,t)
-\sum\limits_{l=0}^{k-2}t\nabla_p\nabla_p\alpha_lG_l\right)
-\frac{\left<V,E_{p}\right>}{\tau}\left(\nabla_p\alpha_{k-1}(u,x,t)-\sum\limits_{l=0}^{k-2}t
\nabla_p\alpha_{l}G_{l}\right)\\
&-\left(\frac{c_{8}}{H}+\frac{f'}{\tau}\right)
\left(\alpha_{k-1}(u,x,t)-\sum\limits_{l=0}^{k-2}(k-l)t\alpha_{l}G_{l}\right).
\end{split}
\end{eqnarray*}
A direction calculation implies
\begin{eqnarray}\label{nabla-ieq}
|\nabla_p\alpha_{k-1}(u,x,t)|\leq c_{10},\quad
|\nabla_p\nabla_p\alpha_{k-1}(u,x,t)|\leq c_{11}(1+H),
\end{eqnarray}
where the positive constant $c_{10}$ depends on
$|\alpha_{l}|_{C^{1}}$, and the positive constant $c_{11}$ depends
on $|\alpha_{l}|_{C^{2}}$. So
\begin{eqnarray*}
\begin{split}
-\frac{1}{H}&c_{12}\left(\sum\limits_{l=0}^{k-2}|G_{l}|+1\right)(H+1)
-c_{13}\left(\sum\limits_{l=0}^{k-2}|G_{l}|+1\right)\\
\leq&\frac{1}{H}\sum\limits_{p=1}^{n}\left(\nabla_p\nabla_p\alpha_{k-1}(u,x,t)
-\sum\limits_{l=0}^{k-2}t\nabla_p\nabla_p\alpha_lG_l\right)
-\frac{\left\langle
V,E_{p}\right\rangle}{\tau}\left(\nabla_p\alpha_{k-1}(u,x,t)-\sum\limits_{l=0}^{k-2}t
\nabla_p\alpha_{l}G_{l}\right)\\
&-\left(\frac{c_{8}}{H}+\frac{f'}{\tau}\right)
\left(\alpha_{k-1}(u,x,t)-\sum\limits_{l=0}^{k-2}(k-l)t\alpha_{l}G_{l}\right)
\end{split}
\end{eqnarray*}
where the positive constant $c_{12}$ depends on $c_{8}$, $c_{10}$,
the $C^{1}$ bound of $f$, and the positive constant $c_{13}$ depends
on $c_{8}$, $c_{11}$, the $C^{1}$ bound of $f$. Then, together with
the fact $|A|^{2}\geq\frac{1}{n}H^{2}$, we have
\begin{eqnarray*}
\begin{split}
\frac{1}{n}&H\alpha_{k-1}(u,x,t)-\left(\frac{c_{7}}{H}+c_{14}\right)\sum\limits_{i}G^{ii}\\
&\leq\frac{|A|^{2}}{H}\left(\alpha_{k-1}(u,x,t)-\sum\limits_{l=0}^{k-2}(k-l)t\alpha_{l}G_{l}\right)
-\left(\frac{c_{7}}{H}+c_{9}\right)\sum\limits_{i}G^{ii}-\frac{\left\langle
V,E_{p}\right\rangle}{\tau} \sum\limits_{p}G^{ii}\bar{R}_{0iip},
\\
\end{split}
\end{eqnarray*}
where the positive constant $c_{14}$ depends on $c_{9}$, the $C^{0}$
bound of the curvature tensor $\bar{R}$. Combining the fact
$\Sigma_{i}G^{ii}\geq \frac{n-k+1}{k}$ with the above two
inequalities, we have
\begin{eqnarray*}
0\geq\frac{1}{n}H\alpha_{k-1}(u,x,t)
-\left(\frac{c_{7}}{H}+c_{14}\right)-\frac{1}{H}c_{12}\left(\sum\limits_{l=0}^{k-2}|G_{l}|+1\right)(H+1)
-c_{13}\left(\sum\limits_{l=0}^{k-2}|G_{l}|+1\right).
\end{eqnarray*}
Let us divide the rest of the proof into two cases:

 Case I. If $\frac{\sigma_{k}}{\sigma_{k-1}}\leq H^{\frac{1}{k}}$, then
we get from $\alpha_{l}\geq c_{l}$ that
\begin{eqnarray*}
|G_{l}|=\frac{\sigma_{l}}{\sigma_{k-1}}\leq\frac{1}{\alpha_{l}}\left(\frac{\sigma_{k}}
{\sigma_{k-1}}+\alpha_{l}(u,x,t)\right)\leq
c_{15}(H^{\frac{1}{k}}+1),
\end{eqnarray*}
where the positive constant $c_{15}$ depends on $c_{l}$,
$|\alpha_{l}|_{C^{0}}$. Thus, we have a contradiction when $H$ is
large enough, which implies $H\leq C$.

Case II. If $\frac{\sigma_{k}}{\sigma_{k-1}}> H^{\frac{1}{k}}$, then
by Lemma \ref{NM-ieq}, one has
\begin{eqnarray*}
|G_{l}|=\frac{\sigma_{l}}{\sigma_{k-1}} \leq
\frac{\sigma_{l}}{\sigma_{l+1}}\cdot\frac{\sigma_{l+1}}{\sigma_{l+2}}
\cdot\cdot\cdot\cdot\frac{\sigma_{k-2}}{\sigma_{k-1}} \leq
c_{16}\left(\frac{\sigma_{k-1}}{\sigma_{k}}\right)^{k-1-l} \leq
H^{-\frac{k-1-l}{k}},
\end{eqnarray*}
where the constant $c_{16}>0$ depends on $k$. In this case, we can
also derive $H\leq C$ easily.

In sum, the conclusion of Lemma \ref{C2} follows directly by using
the fact $|\lambda_{i}|\leq c_{6}H$.
\end{proof}

\section{Existence}  \label{S6}

In this section, we use the degree theory for nonlinear elliptic
equations developed in \cite{L} to prove Theorem \ref{maintheorem}.

After establishing a prior estimates (see Lemmas \ref{C0 estimate},
\ref{C1 estimate} and \ref{C2}), we know that the Eq.
\eqref{equation 1.1} is uniformly elliptic. By \cite{E}, \cite{K}
and Schauder estimates, we have
\begin{eqnarray}\label{Ex-1}
|u|_{C^{4,\alpha}(M^{n})}\leq C
\end{eqnarray}
for any $k$-convex solution $\mathcal{G}$ to the equation
\eqref{equation 1.1}. Define
\begin{eqnarray*}
C_{0}^{4,\alpha}(M^{n})=\{u\in C^{4,\alpha}(M^{n}):
\mathcal{G}=\{\left(u(x),x\right)|x\in M^{n}\}~~{\mathrm{is}}~~
{k\mathrm{-convex}}\}.
\end{eqnarray*}
Let us consider the function
\begin{eqnarray*}
F(\cdot;t):C_{0}^{4,\alpha}(M^{n})\rightarrow C^{2,\alpha}(M^{n}),
\end{eqnarray*}
which is defined by
\begin{eqnarray*}
F(u,x,t)=\frac{\sigma_{k}(\kappa(V))}{\sigma_{k-1}(\kappa(V))}-\sum\limits_{l=0}^{k-2}t
\alpha_{l}(u,x)\frac{\sigma_{l}(\kappa(V))}{\sigma_{k-1}(\kappa(V))}-\alpha_{k-1}(u,x,t).
\end{eqnarray*}
Set
\begin{eqnarray*}
\mathcal{O}_{R}=\{u\in C_{0}^{4,\alpha}(M^{n}):|u|_{C^{4,\alpha}(M^{n})}<R\},
\end{eqnarray*}
which clearly is an open set in $C_{0}^{4,\alpha}(M^{n})$. Moreover,
if $R$ is sufficiently large, $F(u,x,t)=0$ does not have solution on
$\partial\mathcal{O}_{R}$ by the priori estimate established in
\eqref{Ex-1}. Therefore, the degree
$\deg\left(F(\cdot;t),\mathcal{O}_{R},0\right)$ is well-defined for
$0\leq t\leq 1$. Using the homotopic invariance of the degree, we
have
\begin{eqnarray*}
\deg(F(\cdot;1),\mathcal{O}_{R},0)=\deg(F(\cdot;0),\mathcal{O}_{R},0).
\end{eqnarray*}
Lemma \ref{uni-sol} shows that $u=u_{0}$ is the unique solution to
the above equation for $t=0$. By direct calculation, one has
\begin{eqnarray*}
F(su_{0},x;0)=[1-\varphi(su_{0})]\frac{\sigma_{k}(e)}{\sigma_{k-1}(e)}\frac{f'(su_{0})}{f(su_{0})}.
\end{eqnarray*}
Using the fact $\varphi(u_{0})=1$, we have
\begin{eqnarray*}
\delta_{u_{0}}F(u_{0},x;0)=\frac{d}{ds}\Bigg{|}_{s=1}F(su_{0},x;0)=
-\varphi'(u_{0})\frac{\sigma_{k}(e)}{\sigma_{k-1}(e)}\frac{f'(u_{0})}{f(u_{0})}>0,
\end{eqnarray*}
where $\delta F(u_{0},x;0)$ is the linearized operator of $F$ at
$u_{0}$. Clearly, $\delta F(u_{0},x;0)$ has the form
\begin{eqnarray*}
\delta_{\omega}F(u_{0},x;0)=-a^{ij}\omega_{ij}+b^{i}\omega_{i}
-\varphi'(u_{0})\frac{\sigma_{k}(e)}{\sigma_{k-1}(e)}\frac{f'(u_{0})}{f(u_{0})}\omega,
\end{eqnarray*}
where $(a^{ij})_{n\times n}$ is a positive definite matrix. Since
$-\varphi'(u_{0})\frac{\sigma_{k}(e)}{\sigma_{k-1}(e)}\frac{f'(u_{0})}{f(u_{0})}>0$,
then $\delta F(u_{0},x;0)$ is an invertible operator. Therefore,
\begin{eqnarray*}
\deg(F(\cdot;1),\mathcal{O}_{R},0)=\deg(F(\cdot;0),\mathcal{O}_{R},0)=\pm
1,
\end{eqnarray*}
which implies that we can obtain a solution at $t=1$. This finishes
the proof of Theorem \ref{maintheorem}.

\vspace {5 mm}

\section*{Acknowledgments}
This work is partially supported by the NSF of China (Grant Nos.
11801496 and 11926352), the Fok Ying-Tung Education Foundation
(China) and  Hubei Key Laboratory of Applied Mathematics (Hubei
University). The authors sincerely thank Mr. Agen Shang and Prof.
Qiang Tu for sending the digital version of the reference \cite{st1}
to them.

\end{document}